\documentclass[a4paper,11pt]{amsart}

\usepackage[left=2.8cm,right=2.8cm,top=3cm,bottom=3cm]{geometry}
\usepackage{amsmath,amssymb,mathrsfs,amsthm}
\usepackage{bm,esint,mathtools,xcolor}
\usepackage[utf8]{inputenc}
\usepackage[inline]{enumitem}
\usepackage{microtype}

\usepackage{aliascnt}
\usepackage{hyperref,hypcap}
\hypersetup{colorlinks=true,citecolor=red,linkcolor=blue,urlcolor=purple}
\usepackage[capitalize]{cleveref}
\crefformat{equation}{#2(#1)#3}

\theoremstyle{plain}
\newtheorem{theorem}{Theorem}
\newaliascnt{corollary}{theorem}
\newtheorem{corollary}[corollary]{Corollary}
\aliascntresetthe{corollary}
\newaliascnt{lemma}{theorem}
\newtheorem{lemma}[lemma]{Lemma}
\aliascntresetthe{lemma}

\theoremstyle{remark}
\newaliascnt{remark}{theorem}
\newtheorem{remark}[remark]{Remark}
\aliascntresetthe{remark}

\newcommand{\diff}{\mathop{}\!\mathrm{d}}

\DeclareMathOperator{\supp}{supp}

\DeclareMathOperator{\LittleO}{o}

\allowdisplaybreaks

\begin{document}

\title[Resolvent Bounds Imply Controllability]{Resolvent Bounds Imply Observability\\ from Measurable Time Sets}

\author[N.~Burq]{Nicolas Burq}
\address[Nicolas Burq]{Laboratoire de Mathématiques d'Orsay, Université Paris-Saclay and Institut Universitaire de France}
\email{nicolas.burq@universite-paris-saclay.fr}

\author[H.~Zhu]{Hui Zhu}
\address[Hui Zhu]{New York University Abu Dhabi, Division of Science}
\email{hui.zhu@nyu.edu}

\begin{abstract}
    We prove that on a compact Riemannian manifold, resolvent bounds for the Laplace--Beltrami operator imply observability, and thus controllability, for the Schrödinger propagator from time sets of positive Lebesgue measure.

    Applications include almost all cases where observability and controllability hold from time intervals, particularly when the geometric control condition is satisfied or when the manifold is a compact surface of negative curvature.
\end{abstract}

\maketitle

\section{Introduction}

Since the work of Phung and Wang \cite{PhungWang2013} on the observation of heat equations from time sets of positive Lebesgue measure, substantial progress has been made on this topic, relaxing the assumptions on the observation sets and allowed geometries (see e.g., \cite{ApraizEscauriaza2014, ApraizEscauriaza2013, BurqMoyano2023}).
By contrast, though perhaps not surprisingly (since they are notoriously more difficult to handle), much less was known, until recently, about the observability and controllability of other natural evolution equations, such as wave and Schrödinger equations, from measurable sets.

To the best of our knowledge, no observation result is currently known for wave equations.
For Schrödinger equations, the available results are restricted to tori:
Burq and Zworski \cite{BurqZworski2019} (see also Le Balc'h and Martin \cite{LeBalchMartin2023}) proved observability on $\mathbb{T}^2$ from $(0,T) \times \omega$ with $\omega$ of positive measure, and Burq and Zhu \cite{BurqZhu2025tori} extended this result to a general framework on higher dimensional tori, allowing observation sets of the form $\prod_{j=0}^d \omega_j$, where each $\omega_j$ is a subset of $\mathbb{T}$ and has positive measure.

The purpose of this article is to prove, for the Schrödinger equation, the analog of Phung and Wang's result \cite{PhungWang2013}.
We thus fix a \emph{compact} Riemannian manifold $M$, with or without boundary, and let $\Delta$ denote its Laplace--Beltrami operator, subject to Dirichlet or Neumann boundary conditions when a boundary is present.
Let $E \subset \mathbb{R}$ and $\omega \subset M$ be subsets of positive measure.
Recall that the Schrödinger propagator $e^{it\Delta}$, which acts isometrically on $L^2(M)$, is \emph{observable} from $E\times \omega$ if there exists $C>0$ such that, for all $u \in L^2(M)$, we have the observation estimate
\begin{equation}
    \label{eq::OBS}
    \| u\|_{L^2(M)}^2
    \le C \int_E \int_\omega \bigl| e^{it\Delta} u(x) \bigr|^2 \diff x \diff t.
\end{equation}
Our main result is the following.

\begin{theorem}
\label{thm::main-schrodinger}
    Let $\omega \subset M$ be a subset of positive measure.
    Assume there exists $K > 0$ such that, for all $\lambda \ge 1$ and all $u \in H^2(M)$, we have the resolvent bound
    \begin{equation}
    \label{eq::est-resolvent-laplacian}
        K \|u\|_{L^2(M)}
        \le \alpha(\lambda) \|(\lambda+\Delta) u\|_{L^2(M)} + \|u\|_{L^2(\omega)},
    \end{equation}
    where $\alpha(\lambda)>0$ satisfies the limit condition
    \begin{equation}
    \label{eq::limit-condition-alpha}
        \lim_{\lambda \to \infty} \alpha(\lambda)
        = 0.
    \end{equation}
    Then the Schrödinger propagator is observable from $E\times \omega$ for any $E \subset \mathbb{R}$ of positive measure.
\end{theorem}

Also recall that by the Hilbert Uniqueness Method \cite{Lions1988control-v1}, the observation estimate \eqref{eq::OBS} is equivalent to the exact controllability of the Schrödinger equation.
Specifically, if $E \subset [0,T]$, then for all initial and final states $u_0,u_1 \in L^2(M)$, there exists a control $F \in L^2([0,T] \times M)$ which drives $u(0) = u_0$ to $u(T) = u_1$ through the forced Schrödinger equation
\begin{equation*}
    i \partial_t u + \Delta u = \bm{1}_{E \times \omega}F.
\end{equation*}

As another consequence of \cref{thm::main-schrodinger}, in almost all cases where the Schrödinger propagator is observable and controllable from time intervals, these intervals can be replaced by any time set of positive measure. 
More precisely, Burq and Zworski \cite{BurqZworski2004resolvent-control} showed that, in a general functional framework, observation estimates on time intervals follow from resolvent bounds, while Miller \cite{Miller2012} demonstrated that such resolvent bounds are necessary for observability and controllability on \emph{some} time interval. 
Together, these results establish resolvent bounds as a necessary and sufficient condition, modulo a minor discrepancy regarding the lengths of required time intervals (see \cref{remark2}). 
By \cref{thm::main-schrodinger}, we extend Burq and Zworski's result, thereby generalizing previously known observation and control results for Schrödinger equations to measurable time sets.

\begin{remark}
    For future applications, we stated \cref{thm::main-schrodinger} with \emph{measurable} $\omega$. 
    However, existing examples involve only \emph{open} subsets, yielding the following corollaries.
\end{remark}

\begin{corollary}
    If $\omega \subset M$ is an open subset satisfying the geometric control condition, meaning that every geodesic on $M$ reflecting on the boundary according to the laws of geometric optics (generalized bicharacteristics) intersects $\omega$, then for any subset $E \subset \mathbb{R}$ of positive measure, the observation estimate \eqref{eq::OBS} holds.
    Indeed, by Burq and Zworski \cite{BurqZworski2004resolvent-control}, the resolvent bound \eqref{eq::est-resolvent-laplacian} holds in this case with
    \begin{equation*}
        \alpha(\lambda) = \lambda^{-1/2}.
    \end{equation*}
\end{corollary}

\begin{corollary}
    If $M$ is a compact surface of negative curvature and if $\omega \subset M$ is a nonempty open subset, then for any subset $E \subset \mathbb{R}$ of positive measure, the observation estimate \eqref{eq::OBS} holds. Indeed, by Dyatlov and Jin~\cite{DyatlovJin2018semiclassical}, and Dyatlov, Jin and Nonnenmacher \cite{DyatlovJinNonnenmacher2022control}, the resolvent bound \eqref{eq::est-resolvent-laplacian} holds in this case with
    \begin{equation*}
        \alpha(\lambda) = \lambda^{-1/2} \log \lambda.
    \end{equation*}
\end{corollary}

\begin{remark}
\label{remark2}
    In the following two cases, the resolvent bound \eqref{eq::est-resolvent-laplacian} holds with $\alpha(\lambda) = 1$ and cannot hold for $\alpha(\lambda) = \LittleO(1)$ as $\lambda \to \infty$:
    \begin{enumerate}[label=(\roman*)]
        \item $M = (\mathbb{R} / 2\pi \mathbb{Z})^2$ is a two-dimensional torus, and for some $a,b \in (0,2\pi)$,
        \begin{equation*}
            \omega \cap (0,2\pi)^2 = \{(x,y) : a < x < b,\ 0 < y < 2\pi\}.
        \end{equation*}
        A simple example that violates \eqref{eq::est-resolvent-laplacian} with $\alpha(\lambda) = \LittleO(1)$ is $\lambda = k^2$ and $u(x,y) = \chi(x) e^{ik y}$ where $k \in \mathbb{Z}$ and $\chi \in C^\infty(\mathbb{T})$ is a nontrivial function supported outside $(a,b)$.
        \item $M = \{x \in \mathbb{R}^2 : |x| \le 1\}$ is the unit ball in $\mathbb{R}^2$ and, for some $r_0 \in (0,1)$ and $\theta_0 \in (0,2\pi)$,
        \begin{equation*}
            \omega = \{ r e^{i\theta} : r_0 < r < 1,\ 0 <  \theta < \theta_0 \}.
        \end{equation*}
        Examples that violate \eqref{eq::est-resolvent-laplacian} with $\alpha(\lambda) = \LittleO(1)$ can be constructed similarly.
    \end{enumerate}

    In these cases, while observability and controllability are known for arbitrarily small time intervals (see \cite{Jaffard, Komornik, AnantharamanMacia2016}), our proof appears insufficient to derive the observation estimate from time sets of positive measure using the resolvent bound.
    Nevertheless, we recently showed that the observation estimate \eqref{eq::OBS} holds when $\omega$ is open and $E$ has positive measure, and for $d \le 2$, it suffices for $\omega$ to have positive measure \cite{BurqZhu2025tori}. 

    The only other situation where \cref{thm::main-schrodinger} does not apply, yet the Schrödinger propagator remains observable and controllable from time intervals, is the case of the ball \cite{AnantharamanMacia2016}.
    We strongly suspect that revisiting \cite{AnantharamanMacia2016} in the spirit of our work \cite{BurqZhu2025tori} would also allow replacing the time interval by any time set of positive measure.
\end{remark}

\section{Proof}

In this section, we prove \cref{thm::main-schrodinger}, assuming throughout that the resolvent bound \eqref{eq::est-resolvent-laplacian} and the limit condition \eqref{eq::limit-condition-alpha} hold.
Our proof relies on three ingredients:
\begin{enumerate}[label=(\roman*)]
    \item The assumption \eqref{eq::limit-condition-alpha}, which enables us to prove high-frequency observation estimates on time intervals with constants independent of the interval length (see \cref{lem::OBS-SC-short-time});
    \item Measure-theoretical results such as the Lebesgue differentiation theorem and the Severini--Egorov theorem; and
    \item The unitarity of the Schrödinger propagator $e^{it\Delta}$
\end{enumerate}
This approach also applies, at least in the high-frequency regime, to Schrödinger propagators with potentials and to fractional Schrödinger propagators.
However, resolvent bounds for the latter do not appear to be known.

\smallskip

For simplicity, we adopt the following notations:
\begin{itemize}
    \item For any set of parameters $A$, we denote by $C_A$ a finite and positive constant depending solely on $A$.
    We also write $f \lesssim_A g$ if $f \le C_A g$.
    \item We use $\fint$ for averaged integrals.
    That is, if $a < b$ and $g \in L^1(a,b)$, then
    \begin{equation*}
        \fint_a^b g(t) \diff t = \frac{1}{b-a} \int_a^b g(t) \diff t.
    \end{equation*}
\end{itemize}

\subsection{Semiclassical Observability}

First, we establish a semiclassical observation estimate.
Fix a nonnegative and even cutoff function $\phi \in C_c^\infty(\mathbb{R} \setminus \{0\})$ such that $\phi(z)=1$ when $|z| \in (1,2)$.
We also denote $\phi_h(z) = \phi(h^2 z)$ for $h > 0$, where the relation between the semiclassical parameter $h$ and the parameter $\lambda$ in the resolvent bound \eqref{eq::est-resolvent-laplacian} is 
\begin{equation*}
    h = \lambda^{-1/2}.
\end{equation*}

\begin{lemma}
\label{lem::OBS-SC-short-time}
    There exists $C>0$ such that, if $\delta > 0$, then for some $h_\delta > 0$, for all $h \in (0,h_\delta)$, for all $s \in \mathbb{R}$, and for all $u \in L^2(M)$, there holds the semiclassical observation estimate
    \begin{equation}
    \label{eq::OBS-SC-short-time}
        \|\phi_h(\Delta) u\|_{L^2(M)}^2
        \leq C \fint_{s-\delta}^{s+\delta} \|e^{it\Delta} \phi_h(\Delta) u\|_{L^2(\omega)}^2 \diff t.
    \end{equation}
\end{lemma}
\begin{remark}
    The important point in~\eqref{eq::OBS-SC-short-time} is the fact that the high-frequency constant $C$ is uniform for all $\delta > 0$. This is where we use the assumption~\eqref{eq::limit-condition-alpha}.
\end{remark}
\begin{proof}[Proof of \cref{lem::OBS-SC-short-time}]
    By \cite[Theorem~4]{BurqZworski2004resolvent-control}, for some $T > 0$ and $h_0 > 0$, if $h \in (0,h_0)$ and
    \begin{equation}
        \label{eq::BZ-est-condition}
        \delta \ge T \alpha(h^{-2}),
    \end{equation}
    then the semiclassical observation estimate \eqref{eq::OBS-SC-short-time} holds.    
    Indeed, although the referred theorem was stated only for fixed $s$ (in fact $s=0$), the uniformity in $s$ follows from the unitarity of the Schrödinger propagator $e^{it\Delta}$.
    To conclude, notice that by the limit condition \eqref{eq::limit-condition-alpha}, the requirement \eqref{eq::BZ-est-condition} is met when $h$ is small.
\end{proof}

\subsection{Weak Observability}

Next, from the semiclassical observation estimate \eqref{eq::OBS-SC-short-time}, we derive a weak observation estimate using the Littlewood--Paley theory.
Since $\omega$ is only measurable, we will use a temporal version of the theory following \cite{BurqZworski2012control}, which in turn was inspired by the semiclassical approach of Lebeau~\cite{Lebeau1992}.

\begin{lemma}
    There exists $C>0$ such that, if $\delta > 0$, then for some $C_\delta > 0$, for all $s \in \mathbb{R}$, and for all $u \in L^2(M)$, there holds the weak observation estimate
    \begin{equation}
    \label{eq::OBS-weak-short-time}
        \|u\|_{L^2(M)}^2
        \leq C \fint_{s-\delta}^{s+\delta} \|e^{it\Delta} u\|_{L^2(\omega)}^2 \diff t + C_\delta \|u\|_{H^{-2}(M)}^2.
    \end{equation}
\end{lemma}

\begin{proof}
    Let $\chi,\tilde{\chi} \in C_c^\infty(\mathbb{R})$ be such that $\chi(0) \ne 0$ and
    \begin{equation}
        \label{eq::chi-supp-condition}
        \supp \chi \subset \{ \tilde{\chi} = 1\}.
    \end{equation}
    Put $\chi_\delta(t) = \delta^{-1} \chi(t/\delta)$ and $\tilde{\chi}_\delta(t) = \delta^{-1} \tilde{\chi}(t/\delta)$.
    Let $\tilde{\phi} \in C_c^\infty(\mathbb{R} \setminus \{0\})$ be such that $\tilde{\phi} \phi = \phi$, and define $\tilde{\phi}_r(z) = \tilde{\phi}(h^2 z)$.
    We have the conjugation relation
    \begin{equation*}
        e^{i(t+s)\Delta} \phi_h(\Delta) = \phi_h(i\partial_t) e^{i(t+s)\Delta}.
    \end{equation*}
    We also have the commutator relation
    \begin{equation*}
        \chi_\delta \phi_h(i\partial_t) 
        = \tilde{\chi}_\delta\phi_h(i\partial_t) \chi_\delta  \tilde{\phi}_h(i\partial_t) + \tilde{\chi}_\delta[\chi_\delta,\phi_h(i\partial_t)]   \tilde{\phi}_h(i\partial_t).
    \end{equation*}
    Using the two identities above and \cref{lem::OBS-SC-short-time}, we deduce that, for all $h \in (0,h_\delta)$, all $s\in \mathbb{R}$, and all $u \in L^2(M)$, there holds
    \begin{equation}
    \label{eq::OBS-SC-chi}
    \begin{split}
        \|\phi_h(\Delta) u\|_{L^2(M)}^2
        & \lesssim \int_{\mathbb{R}} \|\chi_\delta(t) e^{i(t+s)\Delta} \phi_h(\Delta) u\|_{L^2(\omega)}^2 \diff t \\
        &\lesssim \int_\omega \int_{\mathbb{R}} \bigl|\phi_h(i\partial_t) \bigl(\chi_\delta(t) e^{i(t+s)\Delta} u(x)\bigr) \bigr|^2 \diff t \diff x \\
        & \qquad + \int_\omega \int_{\mathbb{R}} \bigl| \tilde{\chi}_\delta(t)  [\phi_h(i\partial_t),\chi_\delta(t) ] \tilde{\phi}_h(i\partial_t) \bigl(e^{i(t+s)\Delta} u(x)\bigr) \bigr|^2  \diff t \diff x.
    \end{split}
    \end{equation}
    
    By the semiclassical calculus (see e.g., \cite{Zworski2012book}) and the support condition \eqref{eq::chi-supp-condition}, for all $N \ge 1$ and all $h \in (0,1)$, there holds
    \begin{equation}
    \label{eq::commutator-est}
        \bigl\| \tilde{\chi}_\delta [\phi_h(i\partial_t),\chi_\delta] \bigr\|_{L^2(\mathbb{R}) \to L^2(\mathbb{R})}
        \lesssim_{\delta,N} h^N.
    \end{equation}
    Let $h = 2^{-j/2}$ and sum up \eqref{eq::OBS-SC-chi} for all $j \in \mathbb{N}$ satisfying $2^{j/2} < h_\delta$.
    By the Littlewood--Paley theory (we use a spatial theory on the left hand side and a temporal theory on the right hand side), for some $\psi \in C_c^\infty(\mathbb{R})$, there holds
    \begin{equation*}
        \|u\|_{L^2(M)}^2
        \lesssim \int_{\mathbb{R}} \| \chi_\delta(t) e^{i(t+s) \Delta} u \|_{L^2(\omega)}^2 \diff t
        + C_\delta \|u\|_{H^{-2}(M)}^2 + \|\psi(h_\delta^2\Delta) u\|_{L^2(M)}^2.
    \end{equation*}
    To conclude, we use the estimate
    \begin{equation*}
        \|\psi(h_\delta^2\Delta)\|_{H^{-2}(M) \to L^2(M)} \lesssim h_\delta^{-2}.
        \qedhere
    \end{equation*}
\end{proof}

We further derive a weak observation estimate from the measurable time set $E$.

\begin{lemma}
\label{lem::OBS-weak-measurable}
    There exists a subset $J \subset E$ of positive measure, and constants $C>0$ and $\delta_0 > 0$ such that, if $\delta \in (0,\delta_0)$, then for some constant $C_\delta > 0$, for all $s \in J$, and for all $u \in L^2(M)$, there holds the weak observation estimate
    \begin{equation}
    \label{eq::OBS-weak-measurable}
        \|u\|_{L^2(M)}^2
        \leq C \fint_{s-\delta}^{s+\delta} \bm{1}_E(t) \|e^{it\Delta} u\|_{L^2(\omega)}^2 \diff t + C_\delta \|u\|_{H^{-2}(M)}^2.
    \end{equation}
\end{lemma}
\begin{proof}
    It suffices to find $J$ and establish \eqref{eq::OBS-weak-measurable} for $\delta = 2^{-n}$ when $n \in \mathbb{N}$ is large.
    To show this, consider the sequence of functions $(f_n)_{n \ge 0}$ defined by
    \begin{equation*}
        f_n(s) = 1-2^{n-1} |E \cap [s-2^{-n},s+2^{-n}]|.
    \end{equation*}
    By the Lebesgue differentiation theorem, we have the almost everywhere convergence
    \begin{equation}
        \label{eq::f-ae-convergence}
        \lim_{n \to \infty} f_n = \bm{1}_{\mathbb{R} \setminus E}.
    \end{equation}
    By the Severini--Egorov theorem, there exists a subset $J \subset E$ of positive measure, such that the convergence \eqref{eq::f-ae-convergence} is uniform on $J$.
    Therefore,
    \begin{equation}
        \label{eq::OBS-weak-remainder-est}
        \sup_{s \in J} \fint_{s-2^{-n}}^{s+2^{-n}} \bm{1}_{\mathbb{R} \setminus E}(t) \| e^{it \Delta} u\|_{L^2(\omega)}^2 \diff t
        \le \|f_n\|_{L^\infty(J)} \|u\|_{L^2(M)}^2 = \LittleO(1)_{n \to \infty} \|u\|_{L^2(M)}^2.
    \end{equation}
    Therefore, by the estimate \eqref{eq::OBS-weak-remainder-est} and the weak observation estimate \eqref{eq::OBS-weak-short-time}, for all $n \ge 0$, there exists a constant $C_n > 0$, such that
    \begin{equation}
        \label{eq::OBS-weak-dyadic-with-remainder}
        \begin{split}
            \|u\|_{L^2(M)}^2
            & \le C \fint_{s-2^{-n}}^{s+2^{-n}}  \Bigl( 1_E (t) + 1_{\mathbb{R} \setminus E} (t) \Bigr) \|e^{it\Delta} u\|_{L^2(\omega)}^2 \diff t  + C_n \|u\|_{H^{-2}(M)}^2\\
            & \le C \biggl( \fint_{s-2^{-n}}^{s+2^{-n}}  1_E (t) \|e^{it\Delta} u\|_{L^2(\omega)}^2 \diff t + \LittleO(1)_{n \to \infty} \|u\|_{L^2(M)}^2 \biggr) +C_n \|u\|_{H^{-2}(M)}^2.
        \end{split}
    \end{equation} 
    When $n$ is large such that $C \LittleO(1)_{n \to \infty} \le 1/2$, we deduce from the estimate \eqref{eq::OBS-weak-dyadic-with-remainder} that
    \begin{equation*}
        \|u\|_{L^2(M)}^2
        \leq 2C \fint_{s-2^{-n}}^{s+2^{-n}}  1_E (t) \|e^{it\Delta} u\|_{L^2(\omega)}^2 \diff t +2C_n \|u\|_{H^{-2}(M)}^2.
        \qedhere
    \end{equation*}
\end{proof}

\subsection{Uniqueness-Compactness Argument}

We now establish the (strong) observation estimate \eqref{eq::OBS} and thus finish the proof of \cref{thm::main-schrodinger}.

Proceeding with the uniqueness-compactness argument by Bardos, Lebeau and Rauch \cite{BardosLebeauRauch1992}, we assume by contradiction that the observation estimate \eqref{eq::OBS} fails, and choose a sequence $(u_n)_{n \ge 0}$ in $L^2(M)$ which satisfies $\|u_n\|_{L^2(M)} = 1$ and 
\begin{equation}
\label{eq::OBS-vanishing}
    \int_E \|e^{it\Delta} u_n\|_{L^2(\omega)}^2 \diff t
    = \int \bm{1}_E(t) \|e^{it\Delta} u_n\|_{L^2(\omega)}^2 \diff t
    = \LittleO(1)_{n \to \infty}.
\end{equation}

By Fubini's theorem, if $J \subset \mathbb{R}$ has positive measure and $\delta > 0$, then  averaging \eqref{eq::OBS-vanishing} with respect to $s \in J$ yields
\begin{equation}
\label{eq::OBS-vanishing-short-time}
    \fint_J \biggl( \fint_{s-\delta}^{s+\delta} \bm{1}_E(t)  \| e^{it \Delta} u_n\|_{L^2(\omega)}^2 \diff t \biggr) \diff s
    \le \fint_J \bm{1}_E(s)  \| e^{is \Delta} u_n\|_{L^2(\omega)}^2 \diff s
    = \LittleO(1)_{n \to \infty}.
\end{equation}
If in addition $J \subset E$ is given by \cref{lem::OBS-weak-measurable}, then averaging \eqref{eq::OBS-weak-measurable} with respect to $s \in J$ and using the weak observation estimate \eqref{eq::OBS-vanishing-short-time}, we deduce that, for all $\delta \in (0,\delta_0)$ and all sufficiently large $n$, there holds the lower bound
\begin{equation}
\label{eq::weak-OBS-implies-nonvanishing}
    C_\delta \|u_n\|_{H^{-2}(M)}^2
    \gtrsim \|u_n\|_{L^2(M)}^2 - \fint_J \biggl( \fint_{s-\delta}^{s+\delta} \| e^{it \Delta} u_n\|_{L^2(\omega)}^2 \diff t \biggr) \diff s
    \gtrsim 1.
\end{equation}

Up to extracting a subsequence, we can assume that  the sequence $(u_n)_{n \ge 0}$ converges weakly in $L^2(M)$, and thus strongly in $H^{-2}(M)$, to some $u \in L^2(M)$.
The hypothesis \eqref{eq::OBS-vanishing} implies $e^{it\Delta} u(x) = 0$ for almost every $(t,x) \in E \times \omega$.
The lower bound \eqref{eq::weak-OBS-implies-nonvanishing} implies $u \ne 0$.
To finish the proof, it suffices to show that such $u$ does not exist.

\begin{lemma}
\label{lem::unique-continuation}
    If $u \in L^2(M)$ satisfies $e^{it\Delta} u(x) = 0$ for almost every $(t,x) \in E \times \omega$, where $E \subset \mathbb{R}$ and $\omega \subset M$ are subsets of positive measure, then $u = 0$.
\end{lemma}

\begin{remark}
    The proof below follows the argument in our earlier work \cite{BurqZhu2025tori}, which was carried out on tori but extends to general compact manifolds.
    We also note that, by Fubini's theorem, the lemma's assumption can be relaxed to $e^{it\Delta} u(x) = 0$ for almost every $(t,x) \in \Omega$ where $\Omega \subset \mathbb{R} \times M$ has positive measure.
\end{remark}

\begin{proof}[Proof of \cref{lem::unique-continuation}]
    Let $\mathscr{N}$ be the space of all $u \in L^2(M)$ such that $e^{it\Delta} u(x) = 0$ for almost all $(t,x) \in E \times \omega$.
    The lemma is equivalent to saying that $\dim \mathscr{N} = 0$.
    
    First, we claim that, if $u \in \mathscr{N}$, then $e^{it\Delta} u(x) = 0$ for all $t \in \mathbb{R}$ and almost all $x \in \omega$.
    To prove this, let $(\lambda_j)_{j \ge 0}$ be all the distinct eigenvalues of $-\Delta$ with domain $H^1(M)$.
    Expand 
    \begin{equation*}
        u = \sum_{j \ge 0} f_j, \quad  -\Delta f_j = \lambda_j f_j.
    \end{equation*}
    By the orthogonality between eigenfunctions with distinct eigenvalues, we have
    \begin{equation*}
        \int_M \sum_{j \ge 0} |f_j(x)|^2 \diff x
        = \sum_{j \ge 0} \|f_j\|_{L^2(M)}^2
        = \|u\|_{L^2(M)}^2 < \infty.
    \end{equation*}
    Therefore $(f_n(x))_{n \ge 0} \in \ell^2$ for almost all $x \in M$.
    For those $x$, the temporal function
    \begin{equation*}
        t \mapsto e^{it\Delta } u(x) = \sum_{n \ge 0} f_n(x) e^{-it\lambda_n}
    \end{equation*}
    extends to a holomorphic function in the Hardy space $H^2(\mathbb{C}^-)$ where $\mathbb{C}^-$ is the half-plane of complex numbers with negative imaginary part.
    By Fatou's theorem, this extension converges non-tangentially almost everywhere to $e^{it\Delta} u(x)$.
    If in addition $x$ is a point in $\omega$ such that $e^{it\Delta} u(x)$ vanishes almost everywhere on $E$, then by the Lusin--Privalov uniqueness theorem \cite{LusinPrivalov1925} and the fact that $E$ has positive measure, we deduce that $e^{it\Delta} u(x) = 0$ for all $t \in \mathbb{R}$. 
    
    We can now come back to the standard uniqueness-compactness argument used in \cite{BardosLebeauRauch1992}.    
    By the weak observation estimate \eqref{eq::OBS-weak-short-time}, the $L^2(M)$-norm and the $H^{-2}(M)$-norm are equivalent on the space $\mathscr{N}$.
    This fact has two consequences:
    \begin{enumerate}[label=(\roman*)]
        \item First, the space $\mathscr{N}$ is finite dimensional.
        Indeed, the unit ball in $\mathscr{N}$ (with respect to any of the two norms), or more generally any bounded and closed subset of $\mathscr{N}$, is relatively compact.
        The claim follows by the Riesz lemma.
        \item Second, the space $\mathscr{N}$ is invariant by the Laplace--Beltrami operator $\Delta$.
        Indeed, for any $u \in \mathscr{N}$, applying the weak observation estimate \eqref{eq::OBS-weak-short-time} (where we fix any $\delta > 0$) to $v_\epsilon = (e^{i \epsilon  \Delta } u-  u) / \epsilon$ with $\epsilon > 0$, we have
        \begin{equation*}
            \| v_\epsilon \|_{L^2(M)} 
            \le C_\delta \| (e^{i \epsilon \Delta } u- u) / \epsilon \|_{H^{-2}(M)}
            \le C'_\delta \| u\|_{L^2(M)},
        \end{equation*}
        where we used the fact that since $e^{it\Delta} u$ vanishes on $\mathbb{R} \times\omega$, so is $e^{it\Delta} v_\epsilon$ for all $\epsilon > 0$.
        Passing to the limit $\epsilon \to 0$, we deduce that $\Delta u\in L^2$ and $e^{it\Delta} (\Delta u)=  -i\partial_t e^{it\Delta} u$ vanishes on $\mathbb{R} \times \omega$.
        By the definition of $\mathscr{N}$, we have $\Delta u \in \mathscr{N}$.
    \end{enumerate}
    Due to these two consequences, there exists an eigenfunction of $\Delta$ in $\mathscr{N}$.
    By the definition of $\mathscr{N}$, this eigenfunction must vanish almost everywhere on $\omega$, which contradicts the fact that nodal sets of such eigenfunctions have zero Lebesgue measure.
    In fact, their Hausdorff dimensions are at most $d-1$, due to the works \cite{Aronszajn1957unique-continuation,HardtSimon1989nodal}.
\end{proof}

\section*{Acknowledgement}

The authors acknowledge funding from the European Research Council (ERC) under the European Union's Horizon 2020 research and innovation programme (Grant agreement 101097172 -- GEOEDP). 

\bibliographystyle{amsplain}  
\bibliography{references}

\end{document}